\documentclass[11pt,a4paper, DIV=12]{scrartcl}
\usepackage[utf8]{inputenc}
\usepackage[T1]{fontenc}
\usepackage{amsmath}
\usepackage{amssymb}
\usepackage{amscd}
\usepackage{amsthm}
\usepackage{proof}
\usepackage{color}
\usepackage{enumerate}
\usepackage{hyperref}
\setkomafont{sectioning}{\bfseries}

\newtheorem{theorem}{Theorem}[section]
\newtheorem{coro}[theorem]{Corollary}
\newtheorem{lemma}[theorem]{Lemma}
\newtheorem{prop}[theorem]{Proposition}
\newtheorem{rem}[theorem]{\bf Remark}
\newtheorem{definition}[theorem]{Definition}

\theoremstyle{definition}
\newtheorem*{remark}{Remark}

\newtheorem*{example}{Example}

\newcommand{\Hmm}[1]{\leavevmode{\marginpar{\tiny%
$\hbox to 0mm{\hspace*{-0.5mm}$\leftarrow$\hss}%
\vcenter{\vrule depth 0.1mm height 0.1mm width \the\marginparwidth}%
\hbox to 0mm{\hss$\rightarrow$\hspace*{-0.5mm}}$\\\relax\raggedright
#1}}}

\newcommand{\en}{{\cal E}}

\newcommand{\inr}{\mathrm{Inr}}
\newcommand{\vol}{\mathrm{vol}}
\newcommand{\covr}{\mathrm{Covr}}
\newcommand{\diam}{\mathrm{diam}}
\newcommand{\ce}{\mathcal{F}}
\newcommand{\cece}{\mathcal{F}_{\mathrm{c}}}
\newcommand{\lin}{\mathrm{lin}}

\newcommand{\NN}{\mathbb{N}}
\newcommand{\ZZ}{\mathbb{Z}}

\newcommand{\RR}{\mathbb{R}}

\title{Metric  Poincaré type inequalities and lower bounds on the
 infimum of the spectrum for graphs}

\author{D.~Lenz\footnote{ Mathematisches Institut, Friedrich Schiller
Universit{\"a}t Jena, 07743 Jena, Germany, daniel.lenz@uni-jena.de},
M.~Schmidt\footnote{ Mathematisches Institut,
Friedrich Schiller Universit{\"a}t Jena, 07743 Jena, Germany,
schmidt.marcel@uni-jena.de},
P.~Stollmann\footnote{Fakult\"{a}t f\"{u}r Mathematik,  Technische
Universit\"{a}t Chemnitz, D-09107 Chemnitz, Germany,
stollman@math.tu-chemnitz.de }
 }

 \AtBeginDocument{\def\MR#1{}}

\begin{document}
\maketitle

\begin{abstract}
We study metric Poincaré type inequalities on general graphs. We characterize graphs satisfying such inequalities and then turn to the best constants in these inequalities. Invoking suitable metrics we can interpret these constants geometrically as diameters and inradii. Moreover, we can relate them to spectral theory of Laplacians once a probability measure on the graph is chosen. More specifically, we obtain a variational characterization of these constants as infimum over spectral gaps of all Laplacians on the graphs associated to probability measures.
\end{abstract}


\section*{Introduction}
The Poincaré inequality is a fundamental tool in analysis and geometry. For graphs in the $\ell^2$-setting it has received much attention not least as it is equivalent to a spectral gap condition and plays a crucial role in investigation of heat kernel estimates, see e.g.\ \cite{BBK,HLLY,CK,Del,Sal} and references therein. Here, with notation to be explained later in detail we study Poincaré type inequalities on graphs $(X,b)$ of the form
%
\begin{align}
\lvert\sup f - \inf f\rvert^2 \leq c \, \mathcal{E}(f) \tag{MPI} \label{TPI}
\end{align}
%
for all $f$ belonging to the set $\mathcal{D}$ of finite energy functions as well as
\[ \sup \lvert f\rvert^2 \leq c \;\mathcal{E}(f) \]
for all $f$ of finite energy vanishing somewhere. Note that there is no measure involved in these quantities so we think of these inequalities as metric Poincaré inequalities, explaining the~\eqref{TPI} in the formula above.

We study validity and spectral consequences of these inequalities in the three most relevant instances viz.\ for all $f\in\mathcal{D}$, for all $f\in \cece (X)$ (the finitely supported functions) and for all $f\in \mathcal{D}$ vanishing outside a fixed proper subset $\Omega$ of $X$. In all these cases we

\begin{itemize}

\item give geometric characterizations for validity of these inequalities and determine the value of the best constant $c$ in~\eqref{TPI} in terms of suitable metrics;

\item show pure point spectrum for the Laplacians (when an additional ingredient in form of a probability measure is given);

\item prove a variational principle characterizing the best constant in~\eqref{TPI} via taking an infimum over all probability measures of the lowest eigenvalues of the Laplacians.

\end{itemize}

Specifically, the paper is structured as follows:

In Section~\ref{setup} we present the basic setup used in the paper and in Section~\ref{Metrics} we discuss the necessary background on the metrics.

Our first result characterizes those graphs which admit~\eqref{TPI} for all functions $f$ of finite energy (Theorem~\ref{thm-char}). These turn out to be the canonically compactifiable graphs. Such graphs have recently been brought forward~\cite{GHKLW} as graphs with strong intrinsic compactness features and many claims to serve as discrete analogues to open relatively compact subsets of Euclidean spaces, see~\cite{KS,HKSW,KLSS} as well for subsequent studies supporting this point of view. Our second result connects the first result with spectral theory. It shows that in this case the best constant $c$ in~\eqref{TPI} is given via a variational formula by taking the infimum over the second Neumann eigenvalues for all probability measures of full support on $X$ (Theorem~\ref{thm-computing}). As a corollary we obtain a lower bound for the infimum of the spectrum in terms of the diameter of the graph, which is - even for finite combinatorial graphs - better (by the factor $4$) than the usual textbook bounds (Corollary~\ref{cor-textbook}). Moreover, we show that the uniform (in the probability measure) bound to the second Neumann eigenvalue is actually equivalent to a pointwise such bound (Corollary \ref{cor-point-equiv-unif}). 
Along the way we also give a geometric characterization of the best constant in terms of a diameter of the graph with respect to the resistance metric (Proposition~\ref{prop-char-geom-fa}). All these results can be found in Section~\ref{Global}.

Section~\ref{Global-c-c} then deals with~\eqref{TPI} for functions with finite support. We can characterize graphs admitting such an inequality (Theorem~\ref{thm-ut}). These turn out to be the uniformly transient graphs. Such graphs have already featured in various places in the literature, see e.g.\ \cite{BCG,Win}. A systematic study of certain geometric and spectral theoretic properties has recently been given in~\cite{KLSW}. We then go on to present a variational formula for the best constant in terms of spectral theory which in this case is given by taking the infimum over all lowest eigenvalues for all probability measures of full support on $X$ (Theorem~\ref{char-c-null}). Along the way we also obtain a metric characterization of this constant (Proposition~\ref{prop-char-geom-fa-null}).

In Section~\ref{Subgraphs} we then study the restriction of~\eqref{TPI} to functions vanishing outside a prescribed proper subset $\Omega$ of $X$. Here again we can characterize geometrically those subsets admitting such an inequality (Theorem~\ref{char-inradius}) and present a variational formula for the best constant in this context (Theorem~\ref{spectral-theory-omega}). To a certain extent this can be seen as a generalization of the earlier two sections. We provide a discussion of this point of view in the finishing remark of Section~\ref{Subgraphs}.

As a byproduct of our considerations and a method of~\cite{KLSW} we also obtain estimates for higher eigenvalues of the Laplacian (with Neumann boundary conditions) on graphs satisfying a Poincaré inequality. This is discussed in Section~\ref{higher}.

As far as methods go, we note that a crucial part of our considerations concerns geometric interpretation of the best constants in terms of diameters and inradii. More specifically our basic approach relies on considering metrics $\delta$ on $X$ with
\[ \lvert f(x) - f(y)\rvert^2 \leq \delta (x,y) \mathcal{E}(f) \]
for all $x,y\in X$ for suitable $f$. A metric Poincaré inequality then follows once the diameter of the graph in this pseudometric is bounded and - with the right pseudometric - this diameter turns out to be the best constant. The corresponding metrics are discussed in Section~\ref{Metrics}. The sharpness of our results relies on exhibiting for each situation the correct metric. Indeed, from the structural point of view it can be seen as a main achievement of the present article to find the right metric for each of these situations. As a result there are three (slightly) different metrics appearing in our considerations. They can all be seen as variants of the well-known resistance metric. All of them are dominated by a pseudometric $d$ which is the generalization of the combinatorial pseudometric to our case. Thus, in all cases we obtain lower bounds on the spectral theory in terms of the pseudometric $d$. To make this transparent we have included some of these results in the corresponding sections.

Our methods are not confined to the setting of discrete graphs. They can be applied equally well to fractals and quantum graphs and, in fact, resistance forms. For convenience of the reader and to ease the presentation we have decided to present them here in the graph setting only. Details for resistance forms  will be provided elsewhere. In this context, we  note that  a version of our Theorem \ref{thm-computing} for quantum graphs has been recently given in \cite{NadNic}.

We close this introduction by highlighting the following two points:

Of course, it would equally be possible to take square roots in the Poincaré inequalities we presented above and this case can be treated along very similar lines (compare comments in subsequent sections). Here, we stick to the above form of the inequalities for two reasons: One reason is that the metrics appearing in this context are close - and in some cases even equal - to the (natural generalization of the) combinatorial metric. So, we obtain particularly convincing geometric interpretations of the best constants. The other reason is that in our dealing with spectral theory we avoid taking square roots of eigenvalues.

Finally, let us note that we allow for rather general graphs in our considerations. More specifically, all the literature we quote is concerned with graphs satisfying some form of local bounds on the vertex degree (either in form of local finiteness of the graph or, more generally, in form of a local summability condition, see~\eqref{S}). We do not need such restrictions for most of our results.

\medskip

\textbf{Note added:} This article was first posted on the arXiv in 2018. Due to various events in their lives  the  authors   did not push for publication at that point of time. However, the arXiv version has since been  cited and used in various other publications.   The present version is a slight extension with an updated bibliography. We also changed the wording slightly from \emph{topological} Poincaré inequalities to \emph{metric} Poincaré inequalities, since the latter term seems more appropriate.

\section{The set-up}\label{setup} 
In this section we present the set-up we deal with.

A weighted graph $(X,b)$ is given by
\begin{itemize}
\item a countable set $X$, finite or infinite;
\item a symmetric $b\colon X\times X\to [0,\infty)$ called the \textit{edge weight function} satisfying $b(x,x)=0$ for all $x\in X$.
\end{itemize}
 Note that our setting is substantially more general than the one usually assumed in the literature dealing with Laplacians on graphs, see e.g.\ \cite{Chung,KL1,Soa,Woe} and references therein, in that for most results we do \textit{not} require any form of local restrictions on the weights $b$. Only in Section~\ref{Global-c-c} and for one later result (Corollary~\ref{corollary:lambda0 with dirichlet inradius}) we assume that $(X,b)$ satisfies the summability condition
\begin{align*}
\sum_{y\in X}b(x,y)<\infty \tag{S} \label{S}
\end{align*}
for all $x \in X$. The vector space of all real valued functions on $X$ is denoted as $\ce (X)$. The set of all functions in $\ce (X)$ with finite support is denoted by $\cece (X)$. Here, the support $\operatorname{supp}(f)$ of a function $f$ is defined as
\[ \operatorname{supp}(f):=\{x\in X \mid f(x) \neq 0\}. \]
We define the set of \textit{functions of finite energy} associated to the graph by
\[ \mathcal{D}:=\{ f\in \ce (X) \mid \sum_{x,y \in X} b(x,y) \lvert f(x) - f(y)\rvert^2 < \infty\} \]
and define the \textit{energy form} $\mathcal{E}$ associated to the graph on $\mathcal{D}\times \mathcal{D}$ via
\[ \mathcal{E} (f,g):= \frac{1}{2} \sum_{x,y \in X} b(x,y) (f(x) - f(y))(g(x) - g(y)). \]
For $f \in \mathcal{D}$ we let $\mathcal{E}(f) := \mathcal{E}(f,f)$ and $\lVert f \rVert_\mathcal{E} = \mathcal{E}(f)^{1/2}$. Then $\lVert \cdot \rVert_\mathcal{E}$ is a semi-norm on $\mathcal D$.  On the space $\ce (X)$ we define the \textit{variational semi-norm}
\[ \lVert\cdot\rVert_{\mathcal{V}} \colon \ce (X)\longrightarrow [0,\infty], \quad \lVert f\rVert_{\mathcal{V}} :=\sup f - \inf f. \]
The set of all bounded functions on $X$ is denoted by $\ell^\infty (X)$ and equipped with the supremum norm $\lVert\cdot\rVert_\infty$ given by $\lVert f\rVert_\infty :=\sup_{x\in X} \lvert f(x)\rvert$. Clearly, $\lVert\cdot\rVert_{\mathcal{V}}$ is finite if and only if $f$ is bounded and in this case we have
\begin{align*}
\lVert f\rVert_\mathcal{V}= \sup f - \inf f
= \sup_{s\in [\inf f,\sup f]} \{ \lVert f - s 1\rVert_\infty\}
= 2\inf_{c\in \RR} \{ \lVert f - c 1\rVert_\infty\}.
\end{align*}
Therefore, $\lVert\cdot\rVert_\mathcal{V}$ is (up to the constant 2) the quotient norm of $\ell^\infty$ modulo the constant functions.

We now turn to operator theory. We will need an additional ingredient viz.\ a measure $m$ on all subsets of $X$. We will assume that the measure has full support (i.e.\ any element of $X$ has positive measure) and is finite on finite sets. Hence, it can be identified with a function $X \to (0,\infty)$. Our basic Hilbert space is
\[ \ell^2(X,m):=\{ f\in \ce (X) \mid \lVert f\rVert_2^2 = \sum_{x\in X}\lvert f(x)\rvert^2m(x)<\infty\}. \]
We will rely on the theory of (nonnegative and closed) forms to provide selfadjoint operators. As we do not assume any local finiteness condition the following discussion may be in order: Our forms will in general not be densely defined in $\ell^2 (X,m)$. Now, the theory of forms is usually discussed under the assumption that the forms are densely defined. However, as is well-known this is not necessary as one can just develop the theory in a - potentially smaller - Hilbert space, namely the closure of the form domain in the original Hilbert space. So, the operators associated to forms below do not necessarily have dense domain in $\ell^2 (X,m)$ but rather are selfadjoint operators in a suitable closed subspace of $\ell^2 (X,m)$.

The form $\mathcal{E}$ induces the closed form $\mathcal{E}^{(N)}$ with domain given by $\mathcal{D} \cap \ell^2 (X,m)$, see e.g.\ Proposition~1.12 in~\cite{Sch}, whose proof carries over to our setting. The associated operator will be denoted by $H^{(N)}_m$. We will often be interested in the situation that $m$ is a probability measure and $\mathcal{D}$ consists of bounded functions. In this case we clearly have $\mathcal{D} \cap \ell^2 (X,m) = \mathcal{D}$.

Whenever $\Omega$ is a nonempty subset of $ X$ we can restrict the form $\mathcal{E}^{(N)}$ to functions in $\mathcal{D} \cap \ell^2 (X,m)$ that vanish outside of $\Omega$. This restriction is a closed form on $\ell^2(\Omega,m)$ and will be denoted by $\mathcal{E}^\Omega$ and the associated operator will be called $H^\Omega$. Another natural closed form coming from $\mathcal{E}$ is $\mathcal{E}^{(D)}$, whose domain is the closure of $\cece (X) \cap \mathcal D$ with respect to the form norm of $\mathcal{E}^{(N)}$. The corresponding operator will be denoted by $H^{(D)}=H^{(D)}_m$.

\begin{remark}
\begin{enumerate}[(a)]
 \item It can happen that $\mathcal D$ and hence the domains of $\mathcal{E}^{(N)}$, $\mathcal{E}^{(D)}$ and $\mathcal{E}^\Omega$ are  very small. Consider e.g.\ the infinite complete graph, i.e.\ $X$ is infinite and $b(x,y) = 1$ for all $x \neq y$. Then $\mathcal D = \mathbb R \cdot 1$ and $\mathcal E = 0$ on $\mathcal D$. For any finite measure  $m$ on $X$ this implies   $H^{(N)}_m  = 0$, which lives on the Hilbert space $\mathbb R \cdot 1$ (the closure of the form domain of $\mathcal{E}^{(N)}$ in $\ell^2(X,m)$). If $\Omega \subsetneq X$ is a proper subset, then $0$ is the only function in $\mathcal D$ vanishing on $X \setminus \Omega$. Hence,  for any measure on $X$ we have $\mathcal E^\Omega = 0$ on $D(\mathcal E^\Omega) = \{0\}$, which implies $H^{\Omega} = 0$ on the trivial Hilbert space $\{0\}$. Similarly, $H^{{(D)}} = 0$ on the trivial Hilbert space.
 \item We could also consider restrictions of $\mathcal{E}^{(D)}$ to $\Omega\subset X$. However, in order to ease the presentation we refrain from giving details for this case as well.
\end{enumerate}
\end{remark}
A basic ingredient for the spectral theory we need is the following simple lemma.
\begin{lemma}\label{trace}
If $m$ is a finite measure on $X$ and $Q$ is a closed form on $\ell^2(X,m)$ and there exists a $C>0$ with
\[ \lVert f\rVert_\infty^2 \leq C\, Q (f) \]
for all $f\in D(Q)$, then the associated operator $A$ has pure point spectrum, $e^{-A}$ is trace class and we have $\lambda_0 \geq \frac{1}{C m(X)}$ for the lowest eigenvalue $\lambda_0$ of $A$.
\end{lemma}
\begin{proof}
Clearly $e^{-\frac{1}{2} A}$ maps into the form domain and hence into $\ell^\infty (X)$. As $m$ is a finite measure, we have an embedding $\ell^\infty (X) \longrightarrow \ell^2 (X,m)$. Hence, $e^{-\frac{1}{2} A}$ factors over $\ell^\infty (X)$. From a factorisation principle we infer that $e^{-\frac{1}{2} A}$ is Hilbert-Schmidt: see Chapters~11.2 and 11.6 in~\cite{DF} as well as the discussion on p.~318 in~\cite{Stollmann-94}. Hence, $e^{-A}$ is trace class and $A$ has pure point spectrum.

As for the last part of the statement we note the obvious inequality
\[ \lVert f\rVert_2^2 \leq m(X) \lVert f\rVert_\infty^2. \qedhere \]
\end{proof}

\section{Metrics and Poincaré inequalities}\label{Metrics}
Our approach to Poincaré inequalities relies on using suitable pseudometrics. The background is discussed in this section.

We consider a graph $(X,b)$ together with a pseudometric $\delta$ on $X$, by which we mean that $\delta\colon X\times X\to [0,\infty)$ is symmetric and satisfies the triangle inequality. We denote by
\[ U_s(x):= \{y\in X \mid \delta(x,y) < s\} \;\text{ and }\; B_s(x):= \{y\in X \mid \delta(x,y) \le s\} \]
the open and closed balls of radius $s$, respectively. We define the \emph{inradius} of $\Omega \subset X$ with $\Omega \neq X$ by
\[ \inr(\Omega):=\sup\{ s>0\mid \exists x\in\Omega: U_s(x)\subset\Omega\}, \]
with the convention $\inr(\Omega) = 0$ if the supremum is taken over an empty set. Similarly, we define the \textit{diameter} of $X$ by
\[ \diam (X):= \sup\{ \delta (x,y) \mid x,y\in X\}. \]
If the pseudometric is not clear from the context, we will write $\inr_\delta$ and $\diam_\delta$.

We say that the pseudometric $\delta$ satisfies a \textit{metric Poincaré inequality} if
\[ \lvert f(x) - f(y)\rvert^2 \leq \delta(x,y) \mathcal{E} (f) \]
holds for all $f\in \mathcal{D}$ and all $x,y\in X$. We will be particularly concerned with two specific metrics satisfying this condition. These will be discussed next.

We start with a pseudometric which can be seen as a direct generalization of the combinatorial metric. An \emph{edge} of the weighted graph $(X,b)$ is a set $\{ x,y\}$ with positive weight $b(x,y)>0$. Denote by $E$ the set of all edges. Clearly, that induces the structure of a combinatorial graph $(X,E)$. A \emph{path} is a finite sequence of edges with nonempty intersections that can most easily be written as $\gamma=(x_0,x_1,\ldots,x_k)$ where $b(x_j,x_{j+1})>0$ for all $j=0,\ldots,k-1$; if we want to specify the endpoints we say that $\gamma$ is a path from $x_0$ to $x_k$. The \emph{length} of such a path $\gamma$ is given by
\[ L(\gamma):=\sum_{j=0}^{k-1}\frac{1}{b(x_{j},x_{j+1})} . \]
In particular the length of an edge $\{x,y\}$ is given by $\frac{1}{b(x,y)}$. To include trivial cases we also allow trivial paths $(x,x)$ from $x$ to $x$ whose length is $0$. We will throughout assume that our graph is \emph{connected} in the sense that every pair of points is connected by a path. The \emph{distance} between $x$ and $y$ is given by
\[ d(x,y):=\inf\{ L(\gamma)\mid \gamma\text{ a path from }x\text{ to }y\} . \]
Clearly, $d$ is symmetric and satisfies the triangle inequality. If $\sup_y b(x,y) <\infty$ for each $x\in X$, then $d$ is a metric. This pseudometric $d$ satisfies a metric Poincaré inequality. For graphs satisfying $\cece (X)\subset \mathcal{D}$ this has been noted in various places, including the recent~\cite{GHKLW}. The proof carries over to our setting. For the sake of completeness we include it next.

\begin{prop}[$d$ satisfies a metric Poincaré inequality]\label{prop-distance}
Let $x,y\in X$ be arbitrary. Then for any path $\gamma = (x_0,\ldots, x_k)$ from $x$ to $y$ and $f\in \mathcal{D}$ the inequality
\[ \lvert f(x) - f(y)\rvert^2 \leq L(\gamma) \sum_{j=0}^{k-1} b(x_{j},x_{j+1}) (f(x_{j})-f(x_{j+1}))^{2} \]
holds. In particular
\[ \lvert f(x) - f(y)\rvert^2 \leq d(x,y) \mathcal{E}(f) \]
is valid.
\end{prop}
\begin{proof}
It suffices to show the first inequality. Take a path $\gamma=(x_{0},\ldots,x_{k})$ from $x$ to $y$. Using the triangle inequality and the Cauchy-Schwarz inequality we can estimate
\begin{align*}
\lvert f(x)-f(y)\rvert&\leq\sum_{j=0}^{k-1}b(x_{j},x_{j+1})^{\frac{1}{2}} \lvert f(x_{j})-f(x_{j+1})\rvert \frac{1}{b(x_{j},x_{j+1})^{\frac{1}{2}}} \\
&\leq L(\gamma)^{1/2} \left(\sum_{j=0}^{k-1}b(x_{j},x_{j+1}) (f(x_{j})-f(x_{j+1}))^{2}\right)^{\frac{1}{2}}.\qedhere
\end{align*}

\end{proof}

We now turn to the \textit{resistance metric} $r$ defined in the following way: Define $r \colon X\times X\longrightarrow [0,\infty]$ such that for any $x,y\in X$ the number $r(x,y)$ is minimal with
\[ \lvert f(x) - f(y)\rvert^2 \leq r(x,y) \mathcal{E}(f) \]
for all $f\in\mathcal{D}$.

\begin{prop}\label{proposition:r pseudometric}
The function $r$ is a pseudometric with $r \leq d$.
\end{prop}

\begin{remark}
For finite graphs it is well known that $r$ is a pseudometric and this can be extended to locally finite graphs (see e.g.\ discussion in~\cite{GHKLW}). However, even for finite graphs the proof is involved and it is far from clear that it carries over to graphs which are not locally finite, let alone to our more general setting. Thus, we provide an alternative argument. This has the advantage that variants of it can be applied to other cases that will be needed later on.
\end{remark}

\begin{proof}
The inequality $r\leq d$ is clear from the previous proposition. In particular, $r$ is finite. Obviously, $r$ is symmetric. It remains to show the triangle inequality. So, let $x,y,z\in X$ be given. Let $f\in\mathcal{D}$ be arbitrary. We have to show
\[ \lvert f(x) - f(z)\rvert^2 \leq (r(x,y) + r(y,z)) \mathcal{E}(f). \]
The case $f(x) = f(z)$ is obvious. Thus, we can assume without loss of generality $f(x) > f(z)$. We distinguish three cases.

\textit{Case 1: $f(y) \geq f(x)$.} In this case we have
\[ \lvert f(x) - f(z)\rvert^2 \leq \lvert f(y) - f(z)\rvert^2 \leq r(y,z) \mathcal{E}(f) \leq (r(x,y) + r(y,z))\mathcal{E}(f). \]

\textit{Case 2: $f(y) \leq f(z)$:} In this case we have
\[ \lvert f(x) - f(z)\rvert^2 \leq \lvert f(x) - f(y)\rvert^2 \leq r(x,y) \mathcal{E}(f) \leq (r(x,y) + r(y,z))\mathcal{E}(f). \]

\textit{Case 3: $f(x) > f(y) > f(z)$:} Without loss of generality we can assume $f(y) =0$ (as otherwise we could shift everything by $- f(y) 1 $). Decomposing $f$ into positive and negative part we obtain
\[ f_+ (x)^2 = (f_+ (x) - f_+ (y))^2 \leq r(x,y) \mathcal{E} (f_+) \]
and hence
\[ f_+ (x) \leq \sqrt{ r(x,y) \mathcal{E}(f_+)} =: T_x \]
and similarly
\[ f_- (z)\leq \sqrt{ r(y,z) \mathcal{E}(f_-)}=:T_z. \]
By the inequality between geometric and arithmetic mean we have
\begin{align*}
T_x T_z &= \sqrt{ r(x,y) \mathcal{E}(f_+)r(y,z) \mathcal{E}(f_-)} \leq \frac{ r(x,y) \mathcal{E}(f_-) + r(y,z) \mathcal{E}(f_+) }{2}.
\end{align*}
Put together this gives
\begin{align*}
\lvert f(x) - f(z)\rvert^2 &= \lvert f_+ (x) + f_- (z)\rvert^2\\
&\leq T_x^2 + 2 T_x T_z + T_z^2\\
&\leq r(x,y) \mathcal{E}(f_+) + (r(x,y) \mathcal{E}(f_-) + r(y,z) \mathcal{E}(f_+)) + r(y,z) \mathcal{E}(f_-)\\
&= (r(x,y) + r(y,z)) (\mathcal{E}(f_+) + \mathcal{E}(f_-))\\
&\leq (r(x,y) + r(y,z)) (\mathcal{E}(f_+) - 2\mathcal{E}(f_+,f_-) + \mathcal{E}(f_-))\\
&= (r(x,y) + r(y,z)) \mathcal{E}(f).
\end{align*}
Here the last inequality follows as $\mathcal{E}(f_+,f_-) \leq 0$ (which can be seen immediately). By the definition of $r$ and as $f$ was arbitrary this shows
\[ r(x,z)\leq r(x,y) + r(y,z) \]
and the proof is finished.
\end{proof}

\begin{coro}
The pseudometric $r$ satisfies a metric Poincaré inequality and for any other pseudometric $\delta$ satisfying such an inequality $r \leq \delta$ holds.
\end{coro}
\begin{proof}
As shown in the previous proposition $r$ is a pseudometric. The remaining claim is clear from the definition.
\end{proof}

As $r$ is the smallest pseudometric admitting a metric Poincaré inequality, we will in the sequel mostly work with $r$ (as this gives the sharpest estimates). However, we emphasize that all estimates will also hold for any other pseudometric $\delta$ admitting a Poincaré inequality. In particular, they also hold for $d$ (which often is much easier to calculate than $r$). In fact, for locally finite trees the metrics $d$ and $r$ agree (see e.g.\ \cite{GHKLW}).

As mentioned already in the introduction we will also be concerned with the companion inequality
\[ \lVert f\rVert_\infty^2 \leq c \; \mathcal{E} (f) \]
for $f\in\mathcal{D}$ vanishing on a prescribed set. In this context, we say that a pseudometric $\delta$ satisfies a \textit{metric Poincaré inequality on $\Omega \subset X$} if
\begin{align}
\lvert f(x) - f(y)\rvert^2 \leq \delta (x,y) \mathcal{E} (f) \tag{MPI$_\Omega$} \label{TPIOmega}
\end{align}
for $f\in \mathcal{D}$ with $f\geq 0$ vanishing outside $\Omega$. To deal with this situation we introduce the following metric. For $\Omega\subset X$ with $\Omega \neq X$ we define $r_\Omega \colon X\times X\longrightarrow [0,\infty]$ such that for any $x,y\in X$ the number $r_\Omega (x,y)$ is minimal with
\[ \lvert f(x) - f(y)\rvert^2 \leq r_\Omega (x,y) \mathcal{E}(f) \]
for all $f\in \mathcal{D}$ with $f\geq 0$ and $\operatorname{supp}(f)\subset \Omega$. Thus, we have
\[ r_\Omega (x,y)= \sup\{ \lvert f(x) - f(y)\rvert^2 \mid f\in \mathcal{D} \text{ with $f\geq 0$, $\mathcal{E}(f) \leq 1$ and $f = 0$ outside $\Omega$}\}. \]

\begin{prop}
For any $\Omega \subset X$ with $\Omega \neq X$ the function $r_\Omega$ is a pseudometric with $ r_\Omega \leq r$.
\end{prop}

\begin{remark}
We are not aware of an appearance of this pseudometric earlier in the literature. On a conceptual level bringing up this pseudometric can be seen as a key step in the present work.
\end{remark}

\begin{proof}
The inequality $r_\Omega \leq r$ is clear and directly gives that $r_\Omega$ takes finite values. That $r_\Omega$ is a pseudometric follows by a slight variant of the argument given for $r$. In fact, Case 1 and Case 2 go through without any changes. In the discussion of Case 3 the decomposition $f = f_+ - f_-$ has to be replaced by the decomposition
\[ f = f_{c,+} + f_{c,-} \]
with $f_{c,+} := (f - c 1)_+$ and $f_{c,-} := \min\{ f, c\} = c 1 - (f - c 1)_-$ for $c = f(y)$. It is then easy to see that $f_{c,+}$ and $f_{c,-}$ both are supported in $\Omega$ and
\[ \mathcal{E}(f_{c,+}, f_{c,-}) = \mathcal{E}( (f - c 1)_+, c 1 - (f - c 1)_-) = -\mathcal{E}( (f - c 1)_+, (f - c 1)_-) \geq 0 \]
holds. Given this inequality the argument can be carried through as in the case of $r$. We omit the details.
\end{proof}

As in the case of $r$ the following is immediate from the construction.

\begin{coro}
The pseudometric $r_\Omega$ satisfies~\eqref{TPIOmega} and for any other pseudometric $\delta$ satisfying such an inequality $r_\Omega \leq \delta$ holds.
\end{coro}

\begin{remark}
As already mentioned in the introduction one could equally well take square roots in the Poincaré inequalities we presented above. It turns out that the square roots of the metrics $r$ and $r_\Omega$ are metrics again. Indeed, we can consider the pseudometric $\varrho$ defined via
\[ \varrho (x,y) :=\sup\{f(x) - f(y) \mid f\in\mathcal{D}, \mathcal{E}(f) \leq 1\}. \]
This pseudometric was introduced by Davies~\cite{Dav} in his study of non-commutative Dirichlet forms and then used in a similar spirit in~\cite{HKT} (see~\cite{KLSW} as well). From its very definition we have for any $x,y\in X$
\[ \lvert f(x) - f(y)\rvert^2 \leq \varrho(x,y)^2 \mathcal{E}(f) \]
for all $f\in \mathcal{D}$ and $\varrho (x,y)^2$ is the smallest number with this property. So, the pseudometric $\varrho$ is the square root of $r$, see e.g.\ \cite{KLSW}. Similar considerations apply to $r_\Omega$ giving rise to the pseudometric $\varrho_\Omega$. This means that all results presented below could also be framed in terms of the corresponding metrics $\varrho$ and $\varrho_\Omega$.
\end{remark}

\section{\texorpdfstring{Global Poincaré inequality on $\mathcal{D}$}{Global Poincare inequality on D}}\label{Global}
In this section we investigate graphs $(X,b)$ satisfying a global metric Poincaré inequality in that there exists a $c\geq 0$ with
\[ \lvert f(x) - f(y)\rvert^2 \leq c \mathcal{E}(f) \]
for all $f\in \mathcal{D}$ and all $x,y\in X$.

The relevance of $r$ in our context comes from the following proposition.
\begin{prop}\label{proposition:cp as diameter}
The diameter with respect to the resistance metric $r$ is given by
\[ \diam_r (X) = \sup\{ \lVert f\rVert_{\mathcal{V}}^2 \mid f\in \mathcal{D} \text{ with } \mathcal{E}(f)\leq 1\}, \]
where the value $\infty$ is possible.
\end{prop}
\begin{proof}
We calculate
\begin{align*}
\diam_r (X) &= \sup \{ r(x,y) \mid x,y \in X\}\\
&= \sup\{ \lvert f(x)- f(y)\rvert^2 \mid x,y\in X, f\in \mathcal{D} \text{ with } \mathcal{E}(f)\leq 1\}\\
&=\sup\{ \lVert f\rVert_{\mathcal{V}}^2 \mid f\in \mathcal{D} \text{ with } \mathcal{E} (f) \leq 1\}.\qedhere
\end{align*}

\end{proof}

\begin{theorem}[Characterizing the validity of a global  Poincaré inequality] \label{thm-char}
For a graph $(X,b)$ the following statements are equivalent:

\begin{enumerate}[(i)]

\item A global Poincaré inequality holds, i.e.\ there exists a $c\geq 0$ with
\[ \lVert f\rVert_{\mathcal{V}}^2 \leq c \, \mathcal{E}(f) \]
for all $f\in \mathcal{D}$.

\item   $\mathcal{D}\subset \ell^\infty$.

\item $\diam_r (X) <\infty$.
\end{enumerate}
\end{theorem}

\begin{proof}
(i) $\Longleftrightarrow$ (iii) is clear from the previous proposition.

(i) $\Longrightarrow$ (ii): This is clear as (i) implies
\[ \lVert f\rVert_{\mathcal{V}}^2 \leq c \mathcal{E}(f) <\infty \]
for all $f\in\mathcal{D}$.

(ii) $\Longrightarrow$ (i) follows from a closed graph argument.

Consider the closed subspace $\lin\{1\} = \{ \lambda 1 \mid \lambda \in\RR\}$ of $\ell^{\infty} (X)$. Then $\ell^{\infty} (X) /\lin\{1\}$ is a Banach space with respect to
\[ \lVert f + \lin\{1\}\rVert_{\mathcal{V}}:=\sup f - \inf f= 2\inf_{c\in \RR} \{ \lVert f - c 1\rVert_\infty\} , \]
since the latter norm agrees with the quotient norm $\lVert f + \lin\{1\}\rVert_{\ell^{\infty} (X) /\lin\{1\}}$.

Moreover it is not hard to see that $\mathcal{D}/\lin\{1\}$ is a Banach space with the norm $\lVert f + \lin\{1\} \rVert_{\mathcal{E}} := \mathcal{E}(f)^\frac12$. This follows e.g.\ from Proposition~1.6 in~\cite{Sch}, whose proof carries over to our setting; see as well the outline of proof below.

By assumption (ii), the identity induces a well-defined linear map
\[ \mathrm{id}\colon \mathcal{D}/\lin\{1\} \to \ell^{\infty} (X) /\lin\{1\} . \]
Thus, the desired estimate follows from the closed graph theorem once we can show that $\mathrm{id}$ is closed, so that it remains to verify that
\begin{align}
[f_n]\stackrel{\lVert\cdot\rVert_\mathcal{V}}{\to}[f]\text{ and } [f_n]\stackrel{\lVert\cdot\rVert_\en}{\to}[g]\Rightarrow [f]=[g], \tag{$\star$} \label{star}
\end{align}
where we abbreviate $[f]:=f+\lin\{1\}$. To this end we fix $o\in X$ and consider
\[ \lambda\colon \ell^{\infty} (X) /\lin\{1\}\to \ell^{\infty} (X), \lambda[f](x):= f(x)-f(o) . \]
Note that $\lambda[f]$ is well--defined and $\lambda$ is a lifting of the quotient map, i.e.\, $[\lambda[f]]=[f]$. The remaining claim~\eqref{star} will be settled once we see that for $\bullet\in\{ \mathcal{V}, \en\}$,
\begin{align*}
[f_n]\stackrel{\lVert\cdot\rVert_\bullet}{\to}[f]\Rightarrow \lambda[f_n]\to\lambda[f]\text{ pointwise}, \tag{$\star \star$} \label{2star}
\end{align*}
and this is what we see next. Clearly, \eqref{2star} holds for $\bullet=\mathcal{V}$ since $\lambda$ is continuous. For $\bullet=\en$ we can use Proposition~\ref{prop-distance}. In fact, let $x\in X$ and choose $y=o$. Then the inequality in the latter Proposition gives
\[ \lvert\lambda[f_n](x)-\lambda[f](x)\rvert^2\le d(x,o)\en(f_n-f) , \]
and this finishes the proof.

The lifting $\lambda$ can be used to prove the completeness of $\mathcal{D}/\lin\{1\}$ with respect to $\lVert\cdot\rVert_\en$ as follows. By what we just proved, for any Cauchy sequence $([f_n])$ in the latter space, the functions $ \lambda[f_n]$ converge pointwise to some $f$. A standard completeness proof settles that $[f]$ is the desired limit.
\end{proof}

\begin{remark}
\begin{enumerate}[(a)]
\item The notion of a \textit{canonically compactifiable graph} is put forward in~\cite{GHKLW} to denote graphs satisfying (ii) and the summability condition \eqref{S}.  For such graphs the equivalence between (ii) and (iii) is already contained in~\cite{GHKLW}.
\item Whenever a probability measure $m$ is given, a global Poincaré inequality clearly implies an $\ell^2$-Poincaré inequality of the form
\[ \lVert f - m(f) 1\rVert_2^2 \leq c \mathcal{E} (f) \]
with $m(f) = \sum f(x) m(x)$ and $\lVert\cdot\rVert_2$ the norm on $\ell^2 (X,m)$ (compare the discussion in the introduction). The validity of such an $\ell^2$-Poincaré inequality for canonically compactifiable graphs has been observed in~\cite{KS}.
\end{enumerate}
\end{remark}

\begin{coro}\label{corollary:pure point spectrum}
If $(X,b)$ satisfies a global Poincaré inequality and $m$ is a probability measure on $X$, then $H^{(N)}_m$ has pure point spectrum.
\end{coro}
\begin{proof}
Pure point spectrum is well known for canonically compactifiable graphs with finite measures~\cite{GHKLW}. Hence, the corollary follows immediately from the previous theorem in the case of locally summable $b$. For the general case we note that the constant function $1$ is clearly an eigenfunction to the eigenvalue $0$. We can then conclude the statement from Lemma~\ref{trace} applied to the restriction of $\mathcal{E}^{(N)}$ to the orthogonal complement of $1$ in $\ell^2(X,m)$. Note that any $f \perp 1$ changes its sign, which leads to $\lVert f \rVert_\infty \leq \lVert f \rVert_{\mathcal V}$.
\end{proof}

We aim at studying the best constant $c$ in the Poincaré inequality. More specifically, whenever $(X,b)$ satisfies
\[ \lVert f\rVert_\mathcal{V}^2 \leq c \mathcal{E}(f) \]
for all $f\in\mathcal{D}$, we can consider the infimum, which is actually a minimum, over all possible such $c$. This minimum will be denoted by $c_P$. For the sake of definiteness we set $c_P = \infty$ if the inequality does not hold for any $c$. Here come simple geometric and functional analytic characterisations of this constant.

\begin{prop}[Geometric and functional analytic  interpretation of $c_P$] \label{prop-char-geom-fa}
The best constant in the global Poincaré inequality is given by
\[c_P = \diam_r (X),\]
where the value $\infty$ is possible. If finite, it is equal to the square of the norm of the embedding $(\mathcal{D}/\lin \{1\} ,\lVert\cdot\rVert_\mathcal{E})\longrightarrow (\ell^\infty (X)/ \lin\{1\}, \lVert\cdot\rVert_{\mathcal{V}})$.

\end{prop}
\begin{proof}
This is immediate from the definitions and the proof of Proposition~\ref{proposition:cp as diameter}.
\end{proof}

\begin{remark}
Clearly, any finite graph has finite diameter with respect to $r$. However, it is not hard to construct infinite graphs with finite diameter. In fact, it is easy to see that any infinite graph with
\[ \kappa:=\sum_{x,y : b(x,y)>0} \frac{1}{b(x,y)} <\infty \]
has finite diameter (bounded by $\kappa$) with respect to $d$ (see~\cite{GHKLW} for further discussion as well). Due to $r \leq d$ it has then finite diameter with respect to $r$ as well.
\end{remark}

Our second main result gives a formula for $c_P$. We need the following proposition as preparation.

\begin{prop}\label{proposition:inequality}
Let $X$ be a countable set and $m$ a probability measure on $X$.
\begin{enumerate}[(a)]
\item Then,
\[ \lVert f\rVert_2^2 \leq \frac{1}{4} \lVert f\rVert_{\mathcal{V}}^2 \]
for any bounded $f$ with zero mean i.e.\ with $\sum_x f(x) m(x) = 0$.
\item If $f$ is bounded and satisfies furthermore $\sup f = - \inf f$ we even have
\[ \sup_m \, \lVert f\rVert_2^2 = \frac{1}{4} \lVert f\rVert_{\mathcal{V}}^2, \]
where the supremum is taken over all probability measures with full support such that $f$ has zero mean.
\end{enumerate}
\end{prop}
\begin{proof}
(a) We decompose $f$ into positive and negative part $f = f_+ - f_-$ and set
\[ S:= \sup_{x\in X} f_+(x) \quad \text{and}\quad I:=\sup_{x\in X} f_-(x) . \]
Hence, $S$ is the supremum of $f$ and $I$ is the negative of the infimum of $f$. So, we clearly have
\[ \lVert f\rVert_{\mathcal{V}}^2 =(S + I)^2. \]
Set
\[ X_+ :=\{x\in X \mid f(x) >0\} \;\text{and}\: X_-:=\{x\in X \mid f(x) < 0\}. \]
Then, the condition on the mean of $f$ can be written as
\[ \sum_{x\in X_+} f_+ (x) m(x) = \sum_{x\in X_-} f_- (x) m(x)=:T \]
Then, with $p:=\sum_{x\in X_+} m(x)$ and $q :=\sum_{x\in X_-} m(x)$ we can estimate $T$ via
\[ 0 \leq T\leq S p \;\: \text{ as well as}\;\: 0 \leq T\leq I q \]
with
\[ p + q \leq 1 \]
(as $m$ is a probability measure and $f$ may well take the value zero). From these estimates we obtain
\[ T^2 \leq S I p q \]
and after taking square roots
\[ T \leq \sqrt{SI} \sqrt{p q}. \]
Now, due to $p + q \leq 1$ we can easily infer $p q\leq \frac{1}{4}$. Combined with the inequality between arithmetic and geometric mean we infer
\begin{align}
T \leq \frac{1}{4} (S + I). \tag{$\ast$} \label{*}
\end{align}
After this preparation we now easily finish the proof:
\begin{align*}
\lVert f\rVert_2^2 &= \sum_{x\in X_+} f_+ (x)^2 m(x) + \sum_{x\in X_-} f_-(x)^2 m(x)\\
&\leq S \sum_{x\in X_+} f_+(x) m(x) + I \sum_{x\in X_-} f_-(x) m(x)\\
(\text{definition of $T$})\;\: &= (S+ I) T\\
\text{\eqref{*}}\;\: &\leq \frac{1}{4} (S+ I)^2\\
&= \frac{1}{4} \lVert f\rVert_{\mathcal{V}}^2.
\end{align*}
This shows (a).

(b) The inequality $\leq$ is clear from (a). Thus, it remains to show $\geq$. Assume first that $f$ attains both its maximum and its minimum in, say, $x_M$ and $x_m$. Consider now the measure $m$ which gives the mass $1/2$ to $x_m$ and $x_M$. Then, a short calculation shows that $ \lVert f\rVert_2^2 =\frac{1}{4} \lVert f\rVert_{\mathcal{V}}^2$. The problem with this argument is that in general we do not know that $f$ attains maximum and minimum and that $m$ will not be supported on the whole of $X$. So, we will have to modify the argument slightly: We choose $x_m$ and $x_M$ such that the values are very close to the supremum and infimum of $f$ and we then distribute a very small mass on the remaining points, calculate the mean on these remaining points (which will be close to zero due to the smallness of the mass on the remaining points) and now give almost the mass $1/2$ and $1/2$ to $x_M$ and $x_m$. We omit more details.
\end{proof}

\begin{theorem}[Computing $c_P$ via $\ell^2$-theory] \label{thm-computing}
If $(X,b)$ satisfies a global Poincaré inequality, the best possible constant is given by
\[ \frac{4}{c_P} = \inf_m \{\text{second lowest eigenvalue of $H^{(N)}_m$} \}, \]
where the infimum runs over all probability measures of full support.
\end{theorem}

\begin{remark}
\begin{enumerate}[(a)]
 \item By Corollary~\ref{corollary:pure point spectrum} the operator $H^{(N)}_m $ has pure point spectrum if $c_P < \infty$. Clearly, the infimum of the spectrum is zero with the constant functions being eigenfunctions.
 \item In the theorem a global Poincar\'{e} inequality is assumed to hold. It is possible to give a formulation without this assumption. Setting
\[ \lambda_1(m) = \inf \{\mathcal{E}(f) \mid f \in \mathcal D \cap \ell^2(X,m), \lVert f - \bar m(f) 1\rVert_2 = 1 \} \]
with $\bar m(f) = m(X)^{-1} \sum_{x \in X} f(x)m(x)$ and using  the conventions $\inf \emptyset = \infty$,   $A/0 = \infty$ and $A / \infty = 0$ for any $A > 0$ one can also phrase the theorem as
\[ \frac{4}{c_P} = \inf_m \lambda_1(m), \]
where the infimum runs over all probability measures of full support.
\end{enumerate}
\end{remark}

\begin{proof}
We denote the second lowest eigenvalue of $H^{(N)}_m$ by $\lambda_1$ (and thereby suppress the dependence on $m$).

{\em Case 1:} $c_P = 0$. Using the global Poincaré inequality, this implies $\mathcal D = \mathbb R \cdot 1$. Hence,  for every finite measure $m$ the operator  $H^{(N)}_m = 0$ living on the Hilbert space $\mathbb R \cdot 1$ (cf. the Remark at the end of Section~\ref{setup}) only has one eigenvalue ($0$ with multiplicity $1$). Therefore, the $\inf$ in the theorem is taken over an empty set.

{\em Case 2:} $c_P > 0$.
We first show ``$\leq$'':    Let $m$ be an arbitrary probability measure on $X$ with full support. Then, the first eigenvalue of $H^{(N)}_m$ is $0$ with the constant functions as eigenfunctions. Consider now an arbitrary $f\in \mathcal{D}$ with $f \perp 1$. Hence, $f$ has zero mean. So, from (a) of the previous proposition we infer
\[ \mathcal{E}(f) \geq \frac{1}{c_P} \lVert f\rVert_{\mathcal{V}}^2 \geq \frac{4}{c_P} \lVert f\rVert_2^2. \]
As $f\in \mathcal{D}$ with $f\perp 1$ was arbitrary we obtain $\lambda_1 \geq \frac{4}{c_P}.$

We now turn to showing ``$\geq$'':  Let $\varepsilon >0$ be arbitrary. Since $c_P$ is the best constant, we can choose an $f\in \mathcal{D}$ with
\[ \mathcal{E}(f) <\frac{1 + \varepsilon}{c_P}  \lVert f\rVert_{\mathcal{V}}^2. \]

Clearly we can replace $f$ by $f - s 1$ for any $s\in\RR$ without changing the last inequality. In particular, we can assume without loss of generality that $\sup f = - \inf f$. Now, by (b) of the previous proposition we can choose a probability measure $m$ of full support on $X$ such that $f$ has zero mean with respect to this measure and
\[ (1+\varepsilon) \lVert f\rVert_2^2 \geq \frac{1}{4} \lVert f\rVert_{\mathcal{V}}^2 \]
holds. Combining these estimates we infer
\[ \mathcal{E}(f) < 4 \frac{(1+\varepsilon)^2}{c_P}\lVert f\rVert_2^2. \]
As $f$ has zero mean we have $f\perp 1$ and the last inequality gives
\[ \lambda_1 \leq \frac{4}{c_P}  (1+\varepsilon)^2 . \]
As $\varepsilon >0$ was arbitrary the desired statement follows.
\end{proof}

\begin{remark}
The constant $c_P$ appears in the denominator in the preceding theorem. This is due to the fact that we have chosen to write the original inequality in the form $\lVert f\rVert_{\mathcal{V}}^2 \leq c \mathcal{E}(f)$ with the constant appearing on the right hand side. Of course, one could do differently. However, it seems that it is usual to have the constant on the right hand side and, moreover, as shown above, it is possible to give a direct geometric meaning of it in form of a diameter.
\end{remark}

\begin{coro}\label{cor-textbook}
Assume $\diam_d (X)< \infty$. Then, $(X,b)$ satisfies a global Poincaré inequality and for any finite measure of full support $m$, the second lowest eigenvalue $\lambda_1$ of $H^{(N)}_m$ is bounded below by
\[ \lambda_1 \geq \frac{4}{m(X)\diam_d (X)}. \]
\end{coro}
\begin{proof}
As $r \leq d$ the assumption implies $\diam_r (X)\leq \diam_d (X) <\infty$ and the equivalent assumptions of Theorem~\ref{thm-char} are satisfied. Moreover, from Proposition~\ref{prop-char-geom-fa} we have
\[ c_P = \diam_r (X) \leq \diam_d (X). \]
Given this the last statement follows from Theorem~\ref{thm-computing} and a simple scaling of the total mass.
\end{proof}

\begin{remark}
That a graph with $\diam_d (X)< \infty$ (satisfying the additional summability condition~\eqref{S}) is canonically compactifiable is already contained in~\cite{GHKLW}. The main statement of the corollary is the bound for the lowest eigenvalues. Even in the case of finite combinatorial graphs with combinatorial distances $d_{\mathrm{comb}}$ this bound is stronger (by the factor $4$) than what can be found in the textbooks, e.g.\ \cite[Lemma~1.9]{Chung}.
\end{remark}

\begin{remark}[Poincaré via $\varrho$]
It is instructive to interpret the Poincaré inequality via the pseudometric $\varrho$. As noted already we have $\varrho^2 = r$. So, the preceding considerations imply $c_P = (\diam_\varrho (X))^2$. It is now natural to define the \textit{radius} of $X$ with respect to $\varrho$ as $R(\varrho) :=\frac{\diam_\varrho (X)}{2}$. Given this we can interpret the quantity $\frac{4}{c_P}$ appearing in Theorem~\ref{thm-computing} as
\[ \lambda_1 \geq \frac{4}{c_P} = \frac{1}{R(\varrho)^2}. \]
\end{remark}

The previous discussion showed that a global Poincaré inequality leads to a spectral gap for $H^{(N)}_m$ that is independent of the choice of $m$ (as long as the total mass is fixed). In the following corollary we show that having a spectral gap for each probability measure already leads to such a uniform bound. For a finite measure $m$ we write
\[ \lambda_1(m) = \inf \{\mathcal{E}(f) \mid f \in \mathcal D \cap \ell^2(X,m), \lVert f - \bar m(f) 1\rVert_2 = 1 \} \]
with $\bar m(f) = \frac{1}{m(X)}\sum_{x \in X}f(x)m(x)$ for the spectral gap of $H^{(N)}_m$. Note that $\mathcal D = \mathbb R \cdot 1$ implies $\lambda_1(m) = \infty$ (through our convention $\inf \emptyset = \infty$).

\begin{coro}\label{cor-point-equiv-unif}
The following assertions are equivalent.
\begin{enumerate}[(i)]
\item A global Poincaré inequality holds for $(X,b)$.
\item For each probability measure (each finite measure) of full support $m$ we have $\lambda_1(m) > 0$.
\item $\inf_m \lambda_1(m) > 0$, where the infimum runs over all probability measures of full support.
\end{enumerate}

\end{coro}

\begin{proof}
(i) $\Rightarrow$ (iii): This is the content of the previous theorem.

(iii) $\Rightarrow$ (ii): Obvious.

(ii) $\Rightarrow$ (i): Assume the global Poincaré inequality does not hold. By Theorem~\ref{thm-char} there exists an unbounded $f \in \mathcal D$. We choose a finite measure of full support $m$ such that $f \in \ell^1(X,m)$ but $f \not \in \ell^2(X,m)$. This is always possible: Choose pairwise different $x_n \in X$, $n \in \mathbb N$, with $|f(x_n)| \geq n$ and set $m(x_n) = \frac{1}{n^2 |f(x_n)|}$. On the other points of $X$ simply choose $m$ small enough.

We show $\lambda_1(m) = 0$. The sequence $g_n = \min\{|f|,n\}$, $n \in \mathbb N$, belongs to $\ell^2(X,m)$ and satisfies $\mathcal E(g_n) \leq \mathcal E(f) < \infty$. Hence, it suffices to show $\lVert g_n - \bar m(g_n) 1\rVert_2 \to \infty$. But this follows from $\lVert g_n \rVert_2 \to \infty$ and
\[ \lVert g_n \rVert_2 \leq \lVert g_n - \bar m(g_n) 1\rVert_2 + \lVert \bar m(g_n) 1\rVert_2 \leq \lVert g_n - \bar m(g_n) 1\rVert_2 + \bar m(|f|) m(X)^{1/2}. \qedhere \]
\end{proof}

\begin{remark}
 It is quite remarkable that a spectral gap for all probability measures implies a uniform spectral gap for all probability measures. The reason for this is the closed graph theorem, which we used in the proof of  Theorem~\ref{thm-char}.
\end{remark}

 Our results show
 \[\lambda_1(m) \geq \frac{4}{m(X)\diam_r (X)} \geq \frac{4}{m(X)\diam_d (X)} \]
 and if we allow  arbitrary measures, the first inequality cannot be improved. More precisely, on each graph we find a finite measure such that $\lambda_1$ is arbitrarily close to $4 / (m(X)\diam_r (X))$. However, in our proofs the measures approximating  this bound are rather degenerate - almost all of their mass is concentrated at two points. Thus, when fixing a class of measures (e.g.\ uniform measures in the case of finite $X$) one may obtain substantially better estimates under additional geometric assumptions on the graph. One such  assumption is a lower bound on the curvature  (we warn the reader that there are a lot of different  curvature notions on graphs).

 Graphs with nonnegative curvature tend to satisfy an estimate of the form
 \[\lambda_1  \geq \frac{C}{(\diam_\delta (X))^2},\]
 where $C$ is a constant and $\delta$ is a pseudometric that is comparable to $d$, see \cite{CY96,CLY14,EF18,HLLY,Mun23,CL24}. Hence, if $\diam_\delta (X)$ is substantially smaller than $m(X)$ (which typically happens if $(X,b)$ is not '1-dimensional' like a line or a cycle), these estimates are far superior to ours. For example \cite[Theorem 3.5.1]{Mun23} shows that if $X$ is finite, $m = 1$ and $b \in \{0,1\}$, and $(X,b,m)$ has nonnegative Ollivier curvature,  then
 \[\lambda_1 \geq \frac{\ln{2}}{(\diam_d (X))^2},\]
 which clearly beats our estimate if $|X| \diam_r (X) \gg \diam_d (X)^2$. For example, hypercubes have nonnegative Ollivier curvature (see e.g.\ \cite{MS23}) and for the $n$-dimensional hypercube $Q_n$ we have $|Q_n| \diam_r (Q_n) \sim 2^{n} / n$  and $(\diam_d (Q_n))^2 = n^2$ (see e.g. \cite{Spi10} for the estimate $\diam_r (Q_n) \sim 1/n$).

 Not every curvature related result can be stated  as easily (the cited papers use  different metrics and different curvatures)  and so it is beyond the scope of this paper to provide a precise comparison with all of them. We would only like to mention some limits of curvature bounds, which are not present in our estimates:
 \begin{enumerate}[(a)]
  \item The constant $C$ typically depends on specific bounds on $b$ and $m$. All of the cited papers  assume either $M := \inf m > 0$ or  $D:= \sup_x \sum_{y \in X}b(x,y) < \infty$ or $B:= \inf\{ b(x,y) \mid x,y \in X \text{ with }b(x,y) > 0\} > 0$ (or all three conditions). Moreover, the constant $C$ degenerates as these assumptions degenerate in the sense that   $C \to 0$, when $M \to 0$ or $B \to 0$ or $D \to \infty$.
  \item If infinite graphs are treated at all, then local finiteness is always assumed.
  \item Nonnegative curvature may fail (e.g.\ trees are typically negatively curved).
 \end{enumerate}
 %
%
%
%
%

\section{\texorpdfstring{Global Poincaré inequalities for $\cece (X)$}{Global Poincare inequalities for F\_{\mathrm{c}}(X)}}\label{Global-c-c}

In this section we assume that $X$ is infinite and $\cece(X)\subset \mathcal{D}$. In this case it is easy to see that the aforementioned summability condition~\eqref{S} is true. The following elementary observation provides further characterizations for this inclusion.

For $x \in X$ we write $\delta_x$ for the function that equals $1$ at $x$ and is $0$ otherwise.
\begin{prop}
Let $(X,b)$ be a weighted graph. Then the following are equivalent:
\begin{enumerate}[(i)]
\item $(X,b)$ satisfies the summability condition~\eqref{S}: $\sum_{y\in X}b(x,y)< \infty$ for all $x\in X$.
\item $\delta_x\in\mathcal{D}$ for all $x\in X$.
\item $\cece (X)\subset \mathcal{D}$.
\item $\cece (X)\cap \mathcal{D}$ is dense in $\cece(X)$ w.r.t.\ $\lVert\cdot\rVert_\infty$.
\end{enumerate}
\end{prop}
\begin{proof}

(i) $\Longleftrightarrow$ (ii) follows from the identity  $\sum_{y\in X}b(x,y)=\en(\delta_x)$.

(ii) $\Longrightarrow$ (iii) $\Longrightarrow$ (iv) is clear.

(iv) $\Longrightarrow$(ii): Fix $x\in X$. By assumption, for any $\varepsilon >0$ there is $\varphi_\varepsilon\in\mathcal{D}$ such that $\lVert \varphi_\varepsilon-\delta_x\rVert_\infty < \varepsilon$. For $\varepsilon\le \frac{1}{4}$ we have that $\varphi_\varepsilon(y)\le \frac14$ for all $y\not=x$ and $\varphi_\varepsilon(x)\ge\frac34$. Clearly, $\psi:=(\varphi_\varepsilon-\frac14)_+\in \mathcal{D}$ and $\psi=\psi(x)\delta_x$ with $\psi(x)>0$, so that $\delta_x=\psi(x)^{-1}\psi\in\mathcal{D}$.
\end{proof}

\begin{remark}
The equivalence of (iii) and (iv) is a version of Lemma~2.1 in~\cite{KL1}, which states that the domain of a regular Dirichlet form on $\ell^2(X,m)$ (w.r.t.\ the discrete topology on $X$) contains $\cece(X)$. By definition $\en^{(D)}$ is a regular Dirichlet form on $\ell^2(X,m)$ if $\cece (X)\cap D(\en^{(D)}) = \cece(X) \cap \mathcal{D}$ is uniformly dense in $\mathcal F_{\mathrm{c}}(X)$ and form norm dense in $D(\en^{(D)})$. Thus, (iv) is precisely the statement that $\en^{(D)}$ is regular (for one (equivalently any) measure) and (iii) states that $\cece(X) \subset D(\en^{(D)})$ (for one  (equivalently any) measure).
\end{remark}

Under the assumption $\cece (X)\subset \mathcal{D}$ it is natural to study whether there is $c \geq 0$ such that
\[ \lVert f\rVert_\mathcal{V}^2 \leq c \mathcal{E}(f) \]
for all $f\in \cece (X)$. We will see in the sequel that analogues of the results of the previous section hold in this case as well. For $f\in \cece (X)$ we clearly have that $\lVert\cdot\rVert_{\mathcal{V}}$ and $\lVert\cdot\rVert_\infty $ are equivalent in the sense that
\[ \lVert f\rVert_\infty \leq \lVert f\rVert_{\mathcal{V}} \leq 2 \lVert f\rVert_\infty \]
for all $f\in \cece (X)$. For this reason we will now restrict attention to the inequality
\begin{align}
\lVert f\rVert_\infty^2 \leq c \mathcal{E} (f) \tag{MPI$_{\mathrm{c}}$} \label{TPIc}
\end{align}
for all $f\in \cece (X)$ and say that \emph{a global Poincaré inequality holds for $\cece (X)$} provided it is satisfied for a suitable constant $c \geq 0$. Graphs with the summability condition \eqref{S} satisfying this inequality are termed \textit{uniformly transient} in~\cite{KLSW} and have been thoroughly studied in this latter reference. The name uniform transience comes from the following characterization of the transience of a graph (i.e.\ transience of the simple random walk): The graph $(X,b)$ is {\em transient} if for each $x \in X$ there exists $c_x \geq 0$ such that
\[ |f(x)|^2 \leq c_x \mathcal E(f) \]
for all $f\in \cece (X)$, see e.g.\ Chapter~6 in \cite{KLW21}.

Again, we can characterize graphs satisfying \eqref{TPIc} in various ways. We will need one more pseudometric to do so. Specifically, we define $r_0$ as
\[ r_0 :=\sup\{ r_\Omega \mid \Omega\subset X \text{ finite}\}. \]
Hence,
\[ r_0 (x,y) =\sup\{ \lvert f(x) - f(y)\rvert^2 \mid f\in \cece (X), f\geq 0, \mathcal{E}(f) \leq 1\}. \]
The relevance of $r_0$ comes from the following proposition.

\begin{prop}\label{prop-equality}
The diameter with respect to $r_0$ is given by
\[ \diam_{r_0} (X) = \sup\{ \lVert f\rVert_\infty^2 \mid f\in \cece (X) \text{ with } \mathcal{E}(f)\leq 1\}, \]
where the value $\infty$ is possible.
\end{prop}
\begin{proof}
We note first that for $f\in \cece (X)$ with $f\geq 0$ the equality $\sup_{x,y} \lvert f(x) - f(y)\rvert^2 = \lVert f\rVert_\infty^2$ holds and that $\mathcal{E} (\lvert f\rvert) \leq \mathcal{E}(f)$ for any $f\in\mathcal{D}$ (as a direct calculation shows). Here and below we write $\cece^+ (X)$ for the set of all non-negative functions in $\cece (X)$. Now, we can argue as follows
\begin{align*}
\diam_{r_0} (X) &= \sup\{\lvert f(x) - f(y)\rvert^2 \mid x,y\in X,\; f\in \cece^+ (X), \mathcal{E} (f) \leq 1\}\\
&= \sup\{ \lVert f\rVert_\infty^2 \mid f\in \cece^+ (X), \mathcal{E} (f) \leq 1\}\\
&= \sup \{ \lVert \lvert f\rvert \rVert_\infty^2 \mid f\in \cece (X), \mathcal{E} (f) \leq 1\}\\
&= \sup\{\lVert f\rVert_\infty^2 \mid f\in \cece (X), \mathcal{E} (f) \leq 1\}. \qedhere
\end{align*}
%
\end{proof}

We write $\mathcal D_0$ for the space of functions $f \in \mathcal D$ for which there exists a sequence $(f_n)$ in $\cece(X)$ with $f_n \to f$ pointwise and $\mathcal E(f_n - f) \to 0$. If the graph is transient, this is simply the closure of $\cece(X)$ in $\mathcal D$ with respect to the seminorm $\lVert \cdot \rVert_{\mathcal E}$ (which is then actually a norm on $\mathcal D_0$).

%
%

\begin{theorem}[Characterizing the validity of the Poincaré inequality for $\cece (X)$]\label{thm-ut}
For a graph $(X,b)$ satisfying the summability condition~\eqref{S} the following assertions are equivalent:

\begin{enumerate}[(i)]
\item A global Poincaré inequality~\eqref{TPIc} holds for $\cece (X)$.

\item The set $\mathcal{D}_0$ is contained in $\ce_0 (X):=\overline{\cece(X)}^{\lVert\cdot\rVert_{\infty}}$.

\item $\diam_{r_0} (X) <\infty$.

\end{enumerate}
\end{theorem}
\begin{proof}
Since $(X,b)$ satisfies~\eqref{S}, the equivalence between (i) and (ii) is already contained in~\cite{KLSW}. The equivalence of (i) and (iii) follows immediately from the previous proposition.
\end{proof}

Again, we may ask for the best constant $c_P^0$ in the Poincaré inequality~\eqref{TPIc} (with $c_P^0  = \infty$ if it does not hold for any constant). Here, the geometric and the functional analytic description is as follows.

\begin{prop}[Geometric and functional analytic interpretation of $c_P^0$]\label{prop-char-geom-fa-null}
The best constant for the Poincaré inequality~\eqref{TPIc} is given by
\[c_P^0 = \diam_{r_0}(X),\]
where the value $\infty$ is possible. If it is finite, it equals the square of the norm of the embedding $(\mathcal{D}_0,\lVert\cdot\rVert_{\mathcal{E}})\longrightarrow (\ce_0 (X),\lVert\cdot\rVert_\infty)$.
\end{prop}
\begin{proof}
This is clear from the proof of Proposition~\ref{prop-equality}.
\end{proof}

As in the previous section we may also give a description of the best constant in terms of eigenvalues.

\begin{theorem}[Computing $c_P^0$ via $\ell^2$-theory]\label{char-c-null}
If $(X,b)$ satisfies a global Poincaré inequality for $\cece(X)$, then for each finite measure of full support $m$ the operator $H^{(D)}_m$ has discrete spectrum and
\[ \frac{1}{c_P^0} = \inf\{ \text{lowest eigenvalue of $H^{(D)}_m$}\}, \]
where the infimum is taken over all probability measures of full support on $X$.
\end{theorem}

\begin{remark}
 Since we assume $\cece(X) \subset  \mathcal D$,  $H^{(D)}_m$ is densely defined and  we have $c_P^0 > 0$ (use $\mathcal E(\delta_x) = \sum_{y \in X} b(x,y) \in (0,\infty)$ because the underlying graph is connected and satisfies the summability condition \eqref{S}). Hence, other than in our spectral characterization of $c_P$, neither side of the equation may take the value $\infty$.
\end{remark}

\begin{proof}
$H^{(D)}_m$ has discrete spectrum: Since $D(\mathcal E^{(D)})$ is the form norm closure of $\cece(X)$, Inequality~\eqref{TPIc} extends to all $f \in D(\mathcal E^{(D)})$. With this at hand discreteness of the spectrum follows from Lemma~\ref{trace}.

``$\leq$'': Let $f\in \cece (X)$ be arbitrary. Then,
\[ \lVert f\rVert_2^2 \leq \lVert f\rVert_\infty^2 \leq c_P^0 \mathcal{E} (f) \]
and the desired inequality follows.

``$\geq$'': Let $\varepsilon >0$ be arbitrary. Then, there exists an $f\in \cece (X)$ with
\[ \lVert f\rVert_\infty^2 \geq \frac{c_P^0}{ 1 + \varepsilon} \mathcal{E}(f). \]
We can now choose a probability measure on $X$ of full support that gives almost all of its mass to the point, where $f(x)^2$ attains its maximum. Then, $\lVert f\rVert_2^2 \geq \frac{1}{1 + \varepsilon} \lVert f\rVert_\infty^2$. Put together we see
\[ \lVert f\rVert_2^2 \geq \frac{c_P^0}{(1+\varepsilon)^2} \mathcal{E}(f). \]
This implies
\[ \frac{(1+\varepsilon)^2}{c_P^0} \geq \lambda_0 \]
for the lowest eigenvalue $\lambda_0$ of $H_m^{(D)}$ and the desired inequality follows.
\end{proof}

\section{Subgraphs of finite inradius  and  finite measure}\label{Subgraphs}

The preceding sections were concerned with full graphs satisfying global Poincaré type inequalities. Here, we show that a subgraph satisfies a global Poincaré type inequality if and only if it has finite inradius, and we compute the best constant in terms of that inradius.

We consider a graph $(X,b)$. A subset $\Omega$ of $X$ is said to satisfy a \textit{global metric Poincaré inequality} if there exists a $c\geq 0$ with
\[ \lVert f\rVert_\infty^2 \leq c\mathcal{E}(f) \]
for all $f\in \mathcal{D}$ with $\operatorname{supp}(f)\subset \Omega$. (As in the previous section validity of this inequality is equivalent to validity of the corresponding inequality with $\lVert\cdot\rVert_{\mathcal{V}}$ instead of $\lVert\cdot\rVert_\infty$.) In this case, we denote the best constant $c$ in this inequality by $c_P^\Omega$. For the sake of definiteness we set $c_P^\Omega := \infty$ in all other cases. Moreover, in the following theorem we use the convention $1/0 = \infty$ and $\inf \emptyset = \infty$.

\begin{theorem}[Computing $c_P^{\Omega}$ via $\ell^2$-theory]\label{spectral-theory-omega}
Let $\Omega$ be a subset of $X$ satisfying a global metric Poincaré inequality. Then, $H^{\Omega}_m$ has pure point spectrum for any probability measure $m$ of full support in $\Omega$. Moreover,
\[ \frac{1}{c_P^{\Omega}} = \inf\{\lambda_0 \mid \text{$\lambda_0$ lowest eigenvalue of $H^{\Omega}_m$}\}, \]
where the infimum is taken over all probability measures of full support on $\Omega$.
\end{theorem}
\begin{proof}
The first statement follows from Lemma~\ref{trace}. If $c_P^\Omega > 0$, the second statement follows along the same lines as the proof of Theorem~\ref{char-c-null}. If $c_P^{\Omega} = 0$ (which may happen, see the end of Section~\ref{setup}), then $H^\Omega = 0$ on the trivial Hilbert space $\{0\}$. In this case, $H^\Omega$ has empty spectrum and so the equality also holds.
\end{proof}

Our next aim is to give a geometric characterization of validity of a metric Poincaré inequality. To do so, we will invoke the pseudometric $r_\Omega$. Recall that it is defined via
\[ r_\Omega (x,y) :=\sup\{ \lvert f(x) - f(y)\rvert^2 \mid f\in \mathcal{D} \text{ with } f\geq 0, \operatorname{supp}(f)\subset \Omega, \mathcal{E}(f)\leq 1\}. \]

\begin{theorem}[Geometric characterization of $c_P^\Omega$] \label{char-inradius}
Let $\Omega$ be a non-empty proper subset in a graph $(X,b)$. The best constant in the global Poincaré inequality on $\Omega$ is given by
\[ c_P^{\Omega} = \inr_{r_\Omega } (\Omega) = \diam_{r_\Omega} (X), \]
where the value $\infty$ is possible.  In particular, $\Omega$ allows for a metric Poincaré inequality if and only if its inradius with respect to $r_\Omega$ is finite.
\end{theorem}

\begin{remark}
It may seem surprising that inradius and diameter of $\Omega$ agree (with respect to $r_\Omega$). The reason is that the pseudometric $r_\Omega$ treats the whole complement of $\Omega$ as one point and this makes any two points close to $X \setminus \Omega$ close to each other (even if they are rather far apart in terms of, say, the pseudometric $r$ on the set $X$).
\end{remark}

\begin{proof}
The equality ${c_P^{\Omega}} = \diam_{r_\Omega} (X)$ can be shown as in the proof of Proposition~\ref{prop-equality} (with $\cece (X)$ replaced by the set of functions in $\mathcal{D}$ which vanish outside of $\Omega$). It remains to show ${c_P^{\Omega}} =\inr_{r_\Omega } (\Omega)$. To do so, we establish two inequalities:

\textit{The inequality $\inr_{r_\Omega} (\Omega) \leq {c_P^{\Omega}}$ holds:} Choose $x\in X$ and $p\in X\setminus \Omega$. Then, for any non-trivial $f\in \mathcal{D}$ with $f\geq 0$ and $f=0$ on $ X\setminus \Omega$ we find
\[ \lvert f(x) - f(p)\rvert^2 = \lvert f(x)\rvert^2 \leq \lVert f\rVert_\infty^2 \leq {c_P^{\Omega}}\mathcal{E}(f). \]
This implies $r_\Omega (x,p) \leq {c_P^{\Omega}}$ and, therefore, $\inr_{r_\Omega} (\Omega) \leq {c_P^{\Omega}}$.

\textit{The inequality $ {c_P^{\Omega}} \leq \inr_{r_\Omega}(\Omega) $ holds:} Choose $\inr_{r_\Omega} (\Omega) < s$. Then for $x\in X$ we can find a $p\in X\setminus \Omega$ with
\[ r_\Omega (x,p) < s. \]
For $f\in \mathcal{D}$ with $f\geq 0$ supported in $\Omega$ and $\mathcal{E} (f) \leq 1$ this gives
\[ \lvert f(x)\rvert^2= \lvert f(x) - f(p)\rvert^2 \leq s \]
and we conclude $\lVert f\rVert_\infty^2 \leq s$. As $f$ was arbitrary, we obtain ${c_P^{\Omega}} \leq s$. This finishes the proof.
\end{proof}

As an application of the abstract results above we now obtain the following corollary.

\begin{coro}[The finite measure case]\label{finite-measure}
Let a non-empty $\Omega\subset X$ with $\Omega \neq X$ be given and assume $m(\Omega)<\infty$ and $\inr_d(\Omega)<\infty$. Then, $H^\Omega$ has pure point spectrum and for its lowest eigenvalue $\lambda_\Omega$ we have
\[ \lambda_\Omega\ge \frac{1}{\inr_d(\Omega) m(\Omega)} . \]
\end{coro}

\begin{remark}
This generalizes a result of~\cite{LSS}, which required a uniform version of~\eqref{S}.
\end{remark}

\begin{proof}
We clearly have $r_\Omega\leq r \leq d$. Now, the desired statements follow from the previous theorem and Lemma~\ref{trace}.
\end{proof}

Optimizing the previous inequality over all finite sets yields an estimate for the first Dirichlet eigenvalue on the whole graph. This is discussed next. We define the \textit{Dirichlet inradius} with respect to $d$ by
\[ \inr_d^D (X) := \sup_{A\subset X \text{ finite}} \inr_d (A). \]
Clearly, $\inr_d^D (X) \leq \diam_d (X)$. Hence, finiteness of the diameter implies finiteness of the Dirichlet inradius. However, as shown by an example below, we may have finiteness of the Dirichlet inradius for graphs with unbounded diameter.

\begin{coro}\label{corollary:lambda0 with dirichlet inradius}
Let $(X,b)$ be a graph satisfying the summability condition \eqref{S} with $\inr_d^D (X) <\infty$. Then, a global Poincaré inequality holds for $\cece(X)$ and for any finite measure of full support $m$ the lowest eigenvalue $\lambda_0$ of $H^{(D)}$ satisfies
\[ \lambda_0 \geq \frac{1}{\inr_d^D (X)m(X)}. \]
\end{coro}
\begin{proof}
The Poincaré inequality follows from $c_P^\Omega = \inr_{r_\Omega}(\Omega) \leq \inr_{d}(\Omega) \leq \inr^D_{d}(X)$, which hold for all finite $\Omega \subset X$.  Moreover, it is straightforward that $\lambda_0 = \inf \{\lambda_\Omega \mid \emptyset \neq \Omega \subset X \text{ finite}\}$. Hence, the statement on the smallest eigenvalue follows from the previous corollary.
\end{proof}

\begin{example}[$\inr^D_d (X) <\infty$ and $ \diam_d(X)= \infty$]
The example is given by a copy of $\ZZ$ in which at each point a weighted copy of $\NN$ is attached such that the weights make the total diameter of each of these copies of $\NN$ not exceed $1$. Specifically, the vertex set $X$ is given as $\ZZ \times \NN$ and there is an edge with weight $1$ between $(n,0)$ and $(n+1,0)$ for any $n\in\ZZ$ and an edge with weight $2^{k+1}$ between $(n,k)$ and $(n,k+1)$ for any $n\in\ZZ$ and $k\in\NN$ (so that this edge has length $2^{-(k+1)}$). Here and in this example $\NN$ contains $0$. Then, any ball of size, say, $2$, contains a set of the form $\{(n,k) : k\in\NN\}$ for some $n\in\ZZ$ and hence has infinitely many points. Thus, the Dirichlet inradius is bounded by $2$. At the same time, for any natural number $n$ the ball around $(0,0)$ of size $n$ does not cover the graph (as clearly it does not contain, e.g.\, $(n+1,0)$). Hence the diameter of the graph is infinite. Let us emphasize that the example is a locally finite tree and hence the metrics $d$ and $r$ agree (see~\cite{GHKLW}). Thus, the above inradius and diameter could also be taken with respect to $r$.
\end{example}

\begin{remark}[Recovering best constants]
The previous considerations deal with proper subsets $\Omega$ of $X$. So, it may come as a surprise that they can actually be used to say something on Poincaré type inequalities in the case $\Omega = X$ as well. In fact, we can (partially) recover results of the previous sections from it. As this is instructive we briefly discuss this:

{\em Recovering the best constant in $\lVert f\rVert_\infty^2 \leq c \mathcal{E} (f)$ for $f\in \cece (X)$:} We have already noticed that the pseudometric $r_0$ appearing in the dealing with this inequality is the supremum over all $r_\Omega$ for finite $\Omega\subset X$. The considerations of this section give another way to obtain the best constant $c_P^0 = \diam_{r_0} (X)$. In fact, it can then be seen as a form of 'intrinsic Dirichlet inradius' as due to the definition of $r_0$ and Theorem~\ref{char-inradius} we have
\begin{align*}
c_P^0 &= \diam_{r_0} (X)\\
(\text{definition of $r_0$})\;\: &= \sup \{ \diam_{r_\Omega} (X) \mid \Omega \subset X \text{ finite}\} \\
(\text{Theorem~\ref{char-inradius}})\;\: &= \sup\{ \inr_{r_\Omega}(\Omega) \mid \Omega \subset X \text{ finite}\}.
\end{align*}

\textit{Recovering $\lVert f\rVert_{\mathcal{V}}^2 \leq c\mathcal{E} (f)$ for all $f\in\mathcal{D}$:} We may define the pseudometric $r'$ via
\[ r' :=\sup \{r_\Omega \mid \Omega \subsetneq X\}. \]
Then, the considerations of this section show that validity of a Poincaré inequality for all $f$ vanishing somewhere is governed by $r'$. In fact, this is already sufficient to obtain the Poincaré inequality for all $f\in\mathcal{D}$: Consider an arbitrary $f\in \mathcal{D}$. If $f$ is constant, there is nothing to show. If $f$ is not constant, we can assume without loss of generality that $f_+$ is not trivial. Then, by replacing $f$ by $g=f + \lambda 1$ with a suitable $\lambda\in \RR$, we can assume that $\lVert g_+\rVert_\infty$ is as close to $\lVert f\rVert_{\mathcal{V}}$ as we wish and at the same time $g_-$ is not trivial. Then, $g_+$ must be zero somewhere and from the Poincaré inequality for $r'$ we find
\[ \lVert g_+\rVert_\infty^2 \leq \diam_{r'} (X) \mathcal{E} (g_+). \]
As $\lVert g_+\rVert_\infty$ is as close to $\lVert f\rVert_{\mathcal{V}}$ as we wish and $\mathcal{E}(g_+) \leq \mathcal{E} (g)$ holds we can then infer
\[ \lVert f\rVert_{\mathcal{V}}^2 \leq\diam_{r'} (X) \mathcal{E}(f). \]
As clearly, $r' \leq r$ we see that the best constant is given by
\[ c_P = \diam_{r'} (X). \]
Indeed, with a similar argument it is not hard to see that $r'$ equals $r$. Here, again this best constant can be seen as a form of 'intrinsic Neumann inradius' as due to the definition of $r'$ and Theorem~\ref{char-inradius} we have
\[ c_P = \diam_{r'} (X) = \sup\{\diam_{r_\Omega}(X) \mid \Omega\subsetneq X\} = \sup \{ \inr_{r_\Omega}(\Omega) \mid \Omega\subsetneq X\}. \]
\end{remark}

\section{Bounds on higher eigenvalues}\label{higher}
For finite measures a global Poincaré inequality  can also be employed to obtain bounds for large eigenvalues of the operator $H^{(N)}$. This is discussed in the present section. We denote by
\[ 0 = \lambda_0 < \lambda_1 \leq \lambda_2 \leq \ldots\leq \lambda_n \leq \ldots \]
the eigenvalues of $H^{(N)}$, counted with multiplicity. The possibility of controlling higher eigenvalues by Poincaré inequalities comes from the following lemma and the considerations in Section~\ref{Subgraphs}.

\begin{lemma}\label{lemma:general eigenvalue estimate}
Let $m$ be a measure of full support on $X$ such that $H^{(N)}$ has pure point spectrum. Let $F \subset X$ be finite with $\delta_x\in\mathcal{D}$ for all $x\in F$. Then
\[ \lambda_{|F| + 1} \geq \lambda_{X \setminus F}, \]
where $\lambda_{X \setminus F}$ is the smallest eigenvalue of $H^{X \setminus F}$ and $|F|$ denotes the cardinality of $F$.
\end{lemma}
\begin{proof}
Let $n = |F|$. Using the min-max principle and the fact that $1$ is an eigenfunction to the eigenvalue $\lambda_0 = 0$, for $F = \{x_1,\ldots,x_n\}$ we obtain
\begin{align*}
\lambda_{|F|+1} &\geq \inf \left\{ \frac{\mathcal{E}(g)}{\lVert g\rVert_2^2} \, \middle|\, 0 \neq g \in D(\mathcal E^{(N)}) \cap \{\delta_{x_1},\ldots,\delta_{x_n},1\}^\perp\right\}\\
&= \inf \left\{ \frac{\mathcal{E}(g)}{\lVert g\rVert_2^2} \,\middle|\, 0\neq g \in D(\mathcal E ^{(N)}) \text{ with } g|_F = 0 \text{ and } g \perp 1\right\}\\
&\geq \lambda_{X \setminus F}. \qedhere
\end{align*}
\end{proof}

Recall that $c_P$ is the best constant $c$ for which the inequality $\lVert f\rVert_{\mathcal V}^2 \leq c \mathcal{E}(f)$ holds for all $f \in \mathcal{D}$ and that $c_P^\Omega$ is the best constant $c'$ for which the inequality $\lVert f\rVert_\infty^2 \leq c' \mathcal{E}(f)$ holds for all $f \in \mathcal{D}$ with $\operatorname{supp}(f) \subset \Omega$. We also use the convention $c_P = \infty$ and $c^\Omega_P = \infty$ if there exists no such constant.

\begin{theorem}
Let $(X,b)$ be a graph satisfying a global Poincaré inequality and let $m$ be a finite measure of full support on $X$. Then $H^{(N)}$ has pure point spectrum.  If $F \subset X$ is finite with $\delta_x\in\mathcal{D}$ for all $x\in F$, then the following holds.
\begin{enumerate}[(a)]
\item
\[ \lambda_{|F|+1} \geq \frac{4}{c_P m(X \setminus F)}. \]
\item If $|F| \geq 1$, then $c^{X \setminus F}_P \leq c_P <\infty$ and
\[ \lambda_{|F| + 1} \geq \frac{1}{c^{X \setminus F}_P m(X \setminus F)}. \]
\end{enumerate}
\end{theorem}
\begin{proof}

(a): As seen in the proof of Lemma~\ref{lemma:general eigenvalue estimate} we have
\[ \lambda_{|F|+1} \geq \inf \left\{ \frac{\mathcal{E}(g)}{\lVert g\rVert_2^2} \,\middle|\, 0\neq g \in D(\mathcal E ^{(N)}) \text{ with } g|_F = 0 \text{ and } g \perp 1\right\}. \]
For $g \in \ell^2(X,m)$ with $g|_F = 0$ and $g \perp 1$ Proposition~\ref{proposition:inequality} yields $\lVert g\rVert^2_2 \frac{4}{m(X \setminus F)} \leq \lVert g\rVert_{\mathcal V}^2$. This inequality combined with the Poincaré inequality for $\mathcal E$ shows the claim.

(b): A function which vanishes at (at least) one point of $X$ satisfies $\lVert f\rVert_\infty \leq \lVert f\rVert_{\mathcal V}$. Therefore, $c^{X \setminus F}_P \leq c_P <\infty$. With this at hand the claim follows from Lemma~\ref{lemma:general eigenvalue estimate} and Theorem~\ref{spectral-theory-omega}.
\end{proof}

\begin{remark}
\begin{enumerate}[(a)]
\item The method of proof is taken from~\cite{KLSW}, where a similar statement as in (a) is proven for the operator $H^{(D)}_m$ for uniformly transient graphs. Note that the graphs there always satisfy the additional summability condition~\eqref{S}.

\item We chose to present both estimates because (a) gives a better constant (which is sharp for $\lambda_1$, cf.\ Theorem~\ref{thm-computing}), while (b) may provide better asymptotics for large eigenvalues.

\item Note that we have defined the Poincaré inequality on proper subsets of $X$ via the supremum norm. If we had chosen to define it via the variational norm, the estimate in (b) would contain the factor $4$ (as the estimate in (a) does).

\item We note that independent of our work related estimates on higher eigenvalues for canonically compactifiable graphs are also contained in~\cite{HKSW}.
\end{enumerate}

\end{remark}

\bibliographystyle{amsplain}
\bibliography{literatur}

\end{document}